\documentclass{article}
\usepackage{amssymb}
\usepackage{amsmath}
\usepackage{amsfonts}

\newtheorem{theorem}{Theorem}[section]

\newtheorem{corollary}{Corollary}[section]

\newtheorem{definition}{Definition}[section]
\newtheorem{example}{Example}[section]

\newtheorem{lemma}{Lemma}[section]

\newtheorem{proposition}{Proposition}[section]
\newtheorem{remark}{Remark}[section]

\newenvironment{proof}[1][Proof]{\noindent\textbf{#1.} }{\ \rule{0.5em}{0.5em}}
\input{tcilatex}
\begin{document}

\date{}
\title{FUZZY SEQUENTIAL TOPOLOGICAL SPACES(FSTS) }
\author{ M. Singha$^{\alpha }$, N. Tamang$^{\beta }$ and S. De Sarkar$%
^{\gamma }$ \\
%EndAName
Department of Mathematics,\\
University of North Bengal,\\
Rajarammohunpur - 734013\\
West Bengal, India.\\
$^{\alpha }$manoranjan\_singha@rediffmail.com\\
$^{\beta }$nita\_anee@yahoo.in\\
$^{\gamma }$suparnadesarkar@yahoo.co.in}
\maketitle

\begin{abstract}
In this paper we define fuzzy sequential topology(FST) on $X$ which is a
subcollection of ($I^{X}$)$^{%
%TCIMACRO{\U{2115} }%
%BeginExpansion
\mathbb{N}
%EndExpansion
}$ satisfying the conditions given in the definition. In this setting many
pleasant properties of a countable number of fuzzy topologies on $X$
associated as components of a FST can be known.\newline


\end{abstract}

\section{Introduction:}

Let $X$ be a non empty set and $I=[0$, $1]$ be the closed unit interval in
the set $%
%TCIMACRO{\U{211d} }%
%BeginExpansion
\mathbb{R}
%EndExpansion
$ of real numbers. Let $A_{f}(s)=\{A_{f}^{n}\}_{n}$ and $B_{f}(s)=%
\{B_{f}^{n}\}_{n}$ be sequences of fuzzy sets in $X$ called fuzzy sequential
sets in $X$ and we define\newline
$i)$ $A_{f}(s)\ \vee B_{f}(s)=\{A_{f}^{n}\vee B_{f}^{n}\}_{n}$(union), 
\newline
$ii)$ $A_{f}(s)\wedge B_{f}(s)=\{A_{f}^{n}\wedge B_{f}^{n}\}_{n}$%
(intersection), \newline
$iii)$ $A_{f}(s)\leq B_{f}(s)~$if and only if~$A_{f}^{n}\leq B_{f}^{n~}~$for
all~$n\in 
%TCIMACRO{\U{2115} }%
%BeginExpansion
\mathbb{N}
%EndExpansion
$, $%
%TCIMACRO{\U{2115} }%
%BeginExpansion
\mathbb{N}
%EndExpansion
$ being the set of positive integers, \newline
$iv)$ $A_{f}(s)\leq _{w}B_{f}(s)$ if and only if there exists $n\in 
%TCIMACRO{\U{2115} }%
%BeginExpansion
\mathbb{N}
%EndExpansion
~$such that~$A_{f}^{n}\leq B_{f}^{n}$, \newline
$v)$ $A_{f}(s)=B_{f}(s)~$if and only if~$A_{f}^{n}=B_{f}^{n}~$for all~$n\in 
%TCIMACRO{\U{2115} }%
%BeginExpansion
\mathbb{N}
%EndExpansion
$, \newline
$vi)$ $A_{f}(s)(x)=\{A_{f}^{n}(x)\}_{n}$, $x\in X$, \newline
$vii)$ $A_{f}(s)(x)\geq _{M}r$ if and only if $A_{f}^{n}(x)\geq r_{n}$ for
all $n\in M$, where $r=\{r_{n}\}_{n}$ is a sequence in $I$. In particular if 
$M=%
%TCIMACRO{\U{2115} }%
%BeginExpansion
\mathbb{N}
%EndExpansion
$, where $%
%TCIMACRO{\U{2115} }%
%BeginExpansion
\mathbb{N}
%EndExpansion
$ is the set of positive integers, we write $A_{f}(s)(x)\geq r$, \newline
$viii)$ $X_{f}^{l}(s)=\{X_{f}^{n}\}_{n}~$where~$l\in I~$and~$X_{f}^{n}(x)=l$%
, for~all~$x\in X$, $n\in 
%TCIMACRO{\U{2115} }%
%BeginExpansion
\mathbb{N}
%EndExpansion
$, \newline
$ix)$ $A_{f}^{c}(s)=\{1-A_{f}^{n}\}_{n}=\{(A_{f}^{n})^{c}\}_{n}$, called
complement of $A_{f}(s)$,\newline
$x)$ a fuzzy sequential set $P_{f}(s)=\{p_{f}^{n}\}_{n}$ is called a fuzzy
sequential point if there exists $x\in X$ and a non zero sequence $%
r=\{r_{n}\}_{n}$\ in $I$ such that

\begin{eqnarray*}
p_{f}^{n}(t) &=&r_{n}\text{, if }t=x\text{, } \\
&=&0\text{, if }t\in X-\{x\}\text{, for all }n\in 
%TCIMACRO{\U{2115} }%
%BeginExpansion
\mathbb{N}
%EndExpansion
\text{.}
\end{eqnarray*}%
If $M$ be the collection of all $n\in 
%TCIMACRO{\U{2115} }%
%BeginExpansion
\mathbb{N}
%EndExpansion
$ such that $r_{n}\neq 0$, then we can write the above expression as%
\begin{eqnarray*}
p_{f}^{n}(x) &=&r_{n}\text{, whenever }n\in M\text{, } \\
&=&0\text{,}\ \text{whenever }n\in 
%TCIMACRO{\U{2115} }%
%BeginExpansion
\mathbb{N}
%EndExpansion
-M\text{.}
\end{eqnarray*}%
The point $x$ is called the support, $M$\ is called base and $r$ is called
the sequential grade of membership of $x$ in the fuzzy sequential point $%
P_{f}(s)$ and we write $P_{f}(s)=(p_{fx}^{M}$, $r)$. If further $M=\{n\}$, $%
n\in 
%TCIMACRO{\U{2115} }%
%BeginExpansion
\mathbb{N}
%EndExpansion
$, then the fuzzy sequential point is called a simple fuzzy sequential point
and it is denoted by $(p_{fx}^{n}$, $r_{n})$. A fuzzy sequential point is
called complete if its base is the set of natural numbers. A fuzzy
sequential point $P_{f}(s)=(p_{fx}^{M}$, $r)$ is said to belong to $A_{f}(s)$
if and only if $P_{f}(s)\leq A_{f}(s)$ and we write $P_{f}(s)\in A_{f}(s)$.
It is said to belong weakly to $A_{f}(s)$, symbolically $P_{f}(s)\in
_{w}A_{f}(s)$ if and only if there exists $n\in M$ such that $%
p_{f}^{n}(x)\leq A_{f}^{n}(x)$. If $R\subseteq M$ and $s$ is the sequence in 
$I$ same to $r$ in $R$ and vanishes outside $R$ then the fuzzy sequential
point $P_{rf}(s)=(p_{fx}^{R}$, $s)$ is called a reduced fuzzy sequential
point of $P_{f}(s)=(p_{fx}^{M}$, $r)$. A sequence $(x$, $L)=\{A_{n}\}_{n}$
of subsets of $X$, where $A_{n}=\{x\}$, for all $n\in L$ and $A_{n}=\Phi =$%
the null subset of $X$, for all $n\in 
%TCIMACRO{\U{2115} }%
%BeginExpansion
\mathbb{N}
%EndExpansion
-L$, is called a sequential point in $X$.\newline

\section{Definitions and Results:}

\begin{definition}
A family $\delta (s)$ of fuzzy sequential sets on a non empty set $X$
satisfying the properties
\end{definition}

$i)$ $X_{f}^{r}(s)\in \delta (s)~$for~all~$r\in \{0$, $1\}$,

$ii)$ $A_{f}(s)$, $B_{f}(s)\in \delta (s)\Rightarrow A_{f}(s)\wedge
B_{f}(s)\in \delta (s)$ and

$iii)$ for any family $\{A_{fj}(s)\in \delta (s)$, $j\in J\}$, $\underset{%
j\in J}{\vee }A_{fj}(s)\in \delta (s)$.\newline
is called a fuzzy sequential topology(FST) on $X$ and the ordered pair $(X$, 
$\delta (s))$ is called fuzzy sequential topological space(FSTS). The
members of $\delta (s)$ are called open fuzzy sequential sets in $X$.
Complement of an open fuzzy sequential set in $X$ is called closed fuzzy
sequential set in $X.$

\begin{definition}
If $\delta _{1}(s)$ and $\delta _{2}(s)$ be two FSTs on $X$ such that $%
\delta _{1}(s)\subset \delta _{2}(s)$, then we say that $\delta _{2}(s)$ is
finer than $\delta _{1}(s)$ or $\delta _{1}(s)$ is weaker than $\delta
_{2}(s)$.
\end{definition}

\begin{proposition}
If $\delta $ be a fuzzy topology(FT) on $X$, then $\delta ^{%
%TCIMACRO{\U{2115} }%
%BeginExpansion
\mathbb{N}
%EndExpansion
}$ forms a FST on $X$.
\end{proposition}

\begin{proof}
Proof is straightforward.
\end{proof}

We may construct different FSTs on $X$ from a given FT $\delta $ on $X$, $%
\delta ^{%
%TCIMACRO{\U{2115} }%
%BeginExpansion
\mathbb{N}
%EndExpansion
}$ is the finest of all these FSTs. Not only that, any FT $\delta $ on $X$
can be considered as a component of some FST on $X$, one of them is $\delta
^{%
%TCIMACRO{\U{2115} }%
%BeginExpansion
\mathbb{N}
%EndExpansion
}$, there are at least countably many FSTs on $X$ weaker than $\delta ^{N}$
of which $\delta $ is a component. One of them is $\delta ^{^{\prime
}}(s)=\{A_{f}^{n}(s)=\{A_{f}^{n}\}_{n}$; $A_{f}^{n}=A$ for all $n\in 
%TCIMACRO{\U{2115} }%
%BeginExpansion
\mathbb{N}
%EndExpansion
$ and $A\in \delta \}$.

\begin{proposition}
If $(X$, $\delta (s))$ is a FSTS, then $(X$, $\delta _{n})$ is a fuzzy
topological space(FTS), where $\delta _{n}$=$\{A_{f}^{n}$; $A_{f}^{n}(s)$ = $%
\{A_{f}^{n}\}_{n}\in \delta (s)\}$, $n\in 
%TCIMACRO{\U{2115} }%
%BeginExpansion
\mathbb{N}
%EndExpansion
$.
\end{proposition}

\begin{proof}
Proof is omitted.
\end{proof}

\begin{definition}
$(X$, $\delta _{n})$, where $n\in 
%TCIMACRO{\U{2115} }%
%BeginExpansion
\mathbb{N}
%EndExpansion
$, is called the $n^{th}$ component FTS of the FSTS $(X$, $\delta (s))$.
\end{definition}

\begin{proposition}
Let $A_{f}^{n}(s)$=$\{A_{f}^{n}\}_{n}$ be an open(closed) fuzzy sequential
set in the FSTS $(X$, $\delta (s))$, then for each $n\in 
%TCIMACRO{\U{2115} }%
%BeginExpansion
\mathbb{N}
%EndExpansion
$, $A_{f}^{n}$ is an open(closed) fuzzy set in $(X$, $\delta _{n})$ but the
converse is not necessarily true.
\end{proposition}

\begin{proof}
Proof of the first part is omitted. For the converse part let us take the
FSTS $(X$, $\delta (s))$ where $X$ is any non empty set and $\delta (s)$=$%
\{X_{f}^{r}(s)$, $r\in I\}$. Let $\{r_{n}\}_{n}$ be a strictly increasing
sequence in $I$ and $A_{f}(s)$=$\{A_{f}^{n}\}_{n}$, where $A_{f}^{n}$=$%
\overline{r_{n}}$ and $\overline{r_{n}}(x)$=$r_{n}$ for all $x\in X$, $n\in 
%TCIMACRO{\U{2115} }%
%BeginExpansion
\mathbb{N}
%EndExpansion
$. Clearly for each $n\in 
%TCIMACRO{\U{2115} }%
%BeginExpansion
\mathbb{N}
%EndExpansion
$, $A_{f}^{n}$ is open fuzzy set in $(X$, $\delta _{n})$ but $A_{f}(s)$=$%
\{A_{f}^{n}\}_{n}$ is not an open fuzzy sequential set in the FSTS $(X$, $%
\delta (s))$.
\end{proof}

\begin{definition}
Fuzzy sequential sets $A_{f}(s)$=$\{A_{f}^{n}\}_{n}$ and $B_{f}(s)$=$%
\{B_{f}^{n}\}_{n}$ are called quasi-coincident, denoted by $%
A_{f}(s)qB_{f}(s) $ if and only if there exists $x\in X$ such that $%
A_{f}^{n}(x)>(B_{f}^{n})^{c}(x)$, whenever $A_{f}^{n}$ and $B_{f}^{n}$ both
are non $\overline{0\text{.}}$ We write $A_{f}(s)\overline{q}B_{f}(s)$ to
say that $A_{f}(s)$ and $B_{f}(s)$ are not quasi-coincident.
\end{definition}

\begin{definition}
Fuzzy sequential sets $A_{f}(s)$=$\{A_{f}^{n}\}_{n}$ and $B_{f}(s)$=$%
\{B_{f}^{n}\}_{n}$ are called weakly quasi-coincident, denoted by $%
A_{f}(s)q_{w}B_{f}(s)$ if and only if there exists $x\in X$ such that $%
A_{f}^{n}(x)>(B_{f}^{n})^{c}(x)$ for some $n\in 
%TCIMACRO{\U{2115} }%
%BeginExpansion
\mathbb{N}
%EndExpansion
$. We write $A_{f}(s)\overline{q}_{w}B_{f}(s)$ to mean that $A_{f}(s)$ and $%
B_{f}(s)$ are not weakly quasi-coincident.
\end{definition}

\begin{definition}
A fuzzy sequential point $P_{f}(s)=(p_{fx}^{M}$, $r)$ is called
quasi-coincident with $A_{f}(s)$=$\{A_{f}^{n}\}_{n}$, denoted by $%
P_{f}(s)qA_{f}(s)$ if and only if $P_{f}^{n}(x)>(A_{f}^{n})^{c}(x)$ for all $%
n\in M$. If $P_{f}(s)=(p_{fx}^{M}$, $r)$ is not quasi-coincident with $%
A_{f}(s)$, then we write $P_{f}(s)\overline{q}A_{f}(s)$.
\end{definition}

\begin{definition}
A fuzzy sequential point $P_{f}(s)=(p_{fx}^{M}$, $r)$ is called weakly
quasi-coincident with $A_{f}(s)$=$\{A_{f}^{n}\}_{n}$, denoted by $%
P_{f}(s)q_{w}A_{f}(s)$ if and only if $P_{f}^{n}(x)>(A_{f}^{n})^{c}(x)$ for
some $n\in M$. If $P_{f}(s)=(p_{fx}^{M}$, $r)$ is not weakly
quasi-coincident with $A_{f}(s)$, then we write $P_{f}(s)\overline{q}%
_{w}A_{f}(s)$. If $P_{f}^{n}(x)>(A_{f}^{n})^{c}(x)$ for some $n\in
L\subseteq M$, then we say that $P_{f}(s)$ is weakly quasi-coincident with $%
A_{f}(s)$ at the sequential point $(x$, $L)$.
\end{definition}

\begin{proposition}
If the fuzzy sequential sets $A_{f}(s)$=$\{A_{f}^{n}\}_{n}$ and $B_{f}(s)$=$%
\{B_{f}^{n}\}_{n}$ are quasi-coincident, then each pair of non $\overline{0}$
fuzzy sets $A_{f}^{n}$ and $B_{f}^{n}$ is also so but the converse is not
necessarily true.
\end{proposition}

\begin{proof}
Proof of the first part is omitted. For the second part let $A_{f}(s)$=$%
\{A_{f}^{n}\}_{n}$ and $B_{f}(s)$=$\{B_{f}^{n}\}_{n}$ be fuzzy sequential
sets on $%
%TCIMACRO{\U{211d} }%
%BeginExpansion
\mathbb{R}
%EndExpansion
$\ where%
\begin{eqnarray*}
A_{f}^{1}(x) &=&\frac{2}{3}\text{, }x\in (-\infty \text{, }0)\text{, } \\
&=&\frac{1}{3}\text{, }x\in \lbrack 0\text{, }\infty )\text{.}
\end{eqnarray*}%
\begin{eqnarray*}
A_{f}^{2}(x) &=&\frac{1}{3}\text{, }x\in (-\infty \text{, }0)\text{, } \\
&=&\frac{2}{3}\text{, }x\in \lbrack 0\text{, }\infty )\text{.}
\end{eqnarray*}%
\begin{equation*}
A_{f}^{n}(x)=\frac{3}{4}\text{, }x\in 
%TCIMACRO{\U{211d} }%
%BeginExpansion
\mathbb{R}
%EndExpansion
\text{ and }n\neq 1\text{, }2\text{.}
\end{equation*}%
\begin{eqnarray*}
B_{f}^{1}(x) &=&\frac{1}{2}\text{, }x\in (-\infty \text{, }0)\text{, } \\
&=&\frac{2}{3}\text{, }x\in \lbrack 0\text{, }\infty )\text{.}
\end{eqnarray*}%
\begin{eqnarray*}
B_{f}^{2}(x) &=&\frac{1}{4}\text{, }x\in (-\infty \text{, }0)\text{, } \\
&=&\frac{3}{7}\text{, }x\in \lbrack 0\text{, }\infty )\text{.}
\end{eqnarray*}%
\begin{equation*}
B_{f}^{n}(x)=\frac{1}{2}\text{, }x\in 
%TCIMACRO{\U{211d} }%
%BeginExpansion
\mathbb{R}
%EndExpansion
\text{ and }n\neq 1\text{, }2\text{.}
\end{equation*}%
Clearly $A_{f}^{n}qB_{f}^{n}$ for all $n\in 
%TCIMACRO{\U{2115} }%
%BeginExpansion
\mathbb{N}
%EndExpansion
$ but $A_{f}(s)\overline{q}B_{f}(s)$.
\end{proof}

\begin{corollary}
A fuzzy sequential point $P_{f}(s)=(p_{fx}^{M}$, $r)$ is quasi-coincident
with a fuzzy sequential set $A_{f}(s)$=$\{A_{f}^{n}\}_{n}$ if and only if $%
P_{f}^{n}$ and $A_{f}^{n}$ are so for each $n\in M$.
\end{corollary}

\begin{proof}
Proof is straightforward.
\end{proof}

\begin{definition}
A fuzzy sequential set $A_{f}(s)$ in the FSTS $(X$, $\delta (s))$ is called
a neighbourhood(in short nbd) of a fuzzy sequential point $P_{f}(s)$ if and
only if there exists $B_{f}(s)\in \delta (s)$ such that $P_{f}(s)\in
B_{f}(s)\leq A_{f}(s)$. A nbd $A_{f}(s)$ is called open if and only if $%
A_{f}(s)\in \delta (s)$.
\end{definition}

\begin{definition}
A fuzzy sequential set $A_{f}(s)$ in the FSTS $(X$, $\delta (s))$ is called
a weak nbd of a fuzzy sequential point $P_{f}(s)$ if and only if there
exists $B_{f}(s)\in \delta (s)$ such that $P_{f}(s)\in _{w}B_{f}(s)\leq
A_{f}(s)$. A weak nbd $A_{f}(s)$ is called open if and only if $A_{f}(s)\in
\delta (s)$.
\end{definition}

\begin{definition}
A fuzzy sequential set $A_{f}(s)$ in the FSTS $(X$, $\delta (s))$ is called
a $Q$-nbd of a fuzzy sequential point $P_{f}(s)$ if and only if there exists 
$B_{f}(s)\in \delta (s)$ such that $P_{f}(s)qB_{f}(s)\leq A_{f}(s)$. A $Q$%
-nbd $A_{f}(s)$ is called open if and only if $A_{f}(s)\in \delta (s)$.
\end{definition}

\begin{definition}
A fuzzy sequential set $A_{f}(s)$ in the FSTS $(X$, $\delta (s))$ is called
a weak $Q$-nbd of a fuzzy sequential point $P_{f}(s)$ if and only if there
exists $B_{f}(s)\in \delta (s)$ such that $P_{f}(s)q_{w}B_{f}(s)\leq
A_{f}(s) $. A weak $Q$-nbd $A_{f}(s)$ is called open if and only if $%
A_{f}(s)\in \delta (s)$.
\end{definition}

\begin{proposition}
$A_{f}(s)\leq _{w}(\leq )B_{f}(s)$ if and only if $A_{f}(s)$ and $%
B_{f}^{c}(s)$ are not (weakly) quasi-coincident. In particular $P_{f}(s)\in
_{w}(\in )A_{f}(s)$ if and only if $P_{f}(s)$ is not (weakly)
quasi-coincident with $A_{f}^{c}(s)$.
\end{proposition}

\begin{proof}
Proof is omitted.
\end{proof}

\begin{proposition}
Let $\{A_{fj}(s)$, $j\in J\}$ be a family of fuzzy sequential sets in $X$.
Then a fuzzy sequential point $P_{f}(s)q_{w}(\vee _{j\in J}A_{fj}(s))$ if
and only if $P_{f}(s)q_{w}A_{fj}(s)$ for some $j\in J$.
\end{proposition}

\begin{proof}
Let $P_{f}(s)q_{w}(\vee _{j\in J}A_{fj}(s))$ where $P_{f}(s)=(p_{fx}^{M}$, $%
r)$ and $A_{fj}(s)=\{A_{fj}^{n}\}_{n}$. \newline
This implies%
\begin{equation*}
p_{f}^{k}(x)+S_{f}^{k}(x)>1\text{ for some }n=k\in M\text{, where}\vee
_{j\in J}A_{fj}(s)=\{S_{f}^{n}\}_{n}\text{.}
\end{equation*}%
Therefore $S_{f}^{k}(x)=1-p_{f}^{k}(x)+\varepsilon _{k}$ where $\varepsilon
_{k}>0$.---------------------\TEXTsymbol{>}$(i)$ \newline
Also $S_{f}^{k}(x)-\varepsilon _{k}<A_{fj}^{k}(x)$ for some $j\in J$%
.------------------------\TEXTsymbol{>}$(ii)$ \newline
From $(i)$ and $(ii)$ we have $p_{f}^{k}(x)+A_{fj}^{k}(x)>1$, that is, $%
P_{f}(s)q_{w}A_{fj}(s)$ for some $j\in J$. Other implication is
straightforward.
\end{proof}

\begin{corollary}
If $P_{f}(s)qA_{fj}(s)$ for some $j\in J$, then $P_{f}(s)q(\vee _{j\in
J}A_{fj}(s))$, where $\{A_{fj}(s)$, $j\in J\}$ is a family of fuzzy
sequential sets in $X$ but not conversely.
\end{corollary}

\begin{proof}
Proof of the first part is omitted. For second part, let $%
A_{fj}(s)=\{A_{fj}^{n}\}_{n}$, $j=1$, $2$ be fuzzy sequential sets in $%
%TCIMACRO{\U{211d} }%
%BeginExpansion
\mathbb{R}
%EndExpansion
$, where%
\begin{eqnarray*}
A_{f1}^{1}(x) &=&0\text{ for all }x\in 
%TCIMACRO{\U{211d} }%
%BeginExpansion
\mathbb{R}
%EndExpansion
-(0\text{, }1)\text{, } \\
&=&\frac{1}{4}\text{ for all }x\in (0\text{, }1)\text{.}
\end{eqnarray*}%
\begin{eqnarray*}
A_{f1}^{2}(x) &=&0\text{ for all }x\in 
%TCIMACRO{\U{211d} }%
%BeginExpansion
\mathbb{R}
%EndExpansion
-(\frac{1}{3}\text{, }\frac{2}{3})\text{, } \\
&=&\frac{2}{3}\text{ for all }x\in (\frac{1}{3}\text{, }\frac{2}{3})\text{.}
\end{eqnarray*}%
\begin{equation*}
A_{f1}^{n}(x)=0=A_{f2}^{n}(x)\text{ for all }x\in 
%TCIMACRO{\U{211d} }%
%BeginExpansion
\mathbb{R}
%EndExpansion
\text{, }n\neq 1\text{, }2
\end{equation*}%
\begin{eqnarray*}
A_{f2}^{1}(x) &=&0\text{ for all }x\in 
%TCIMACRO{\U{211d} }%
%BeginExpansion
\mathbb{R}
%EndExpansion
-(-\frac{1}{2}\text{, }1)\text{, } \\
&=&\frac{1}{3}\text{ for all }x\in (-\frac{1}{2}\text{, }1)\text{.}
\end{eqnarray*}%
\begin{eqnarray*}
A_{f2}^{2}(x) &=&0\text{ for all }x\in 
%TCIMACRO{\U{211d} }%
%BeginExpansion
\mathbb{R}
%EndExpansion
-(-\frac{1}{2}\text{, }2)\text{, } \\
&=&\frac{1}{5}\text{ for all }x\in (-\frac{1}{2}\text{, }2)\text{.}
\end{eqnarray*}%
The fuzzy sequential point $P_{f}(s)=(p_{f0.5}^{M}$, $r)$ where $M=\{1$, $%
2\} $, $r=\{r_{n}\}_{n}$ and $r_{1}=r_{2}=\frac{7}{10}$ is quasi-coincident
with $A_{f1}(s)\vee A_{f2}(s)$ but it is not so with any one of them.
\end{proof}

\begin{definition}
A subfamily $\beta $ of a FST $\delta (s)$ on $X$ is called a base for $%
\delta (s)$ if and only if to every $A_{f}(s)\in \delta (s)$, there exists a
subfamily $\{B_{fj}(s)$, $j\in J\}$ of $\beta $ such that $A_{f}(s)=\vee
_{j\in J}B_{fj}(s)$.
\end{definition}

\begin{definition}
A subfamily $S=\{S_{f\lambda }(s)$; $\lambda \in \Lambda \}$ of a FST $%
\delta (s)$ on $X$ is called a subbase for $\delta (s)$ if and only if $%
\{\wedge _{j\in J}S_{fj}(s)$; $J=$finite subset of $\Lambda \}$ forms a base
for $\delta (s)$.
\end{definition}

\begin{theorem}
A subfamily $\beta $ of a FST $\delta (s)$ on $X$ is a base for $\delta (s)$
if and only if for each fuzzy sequential point $P_{f}(s)$ in $(X$, $\delta
(s))$ and for every open weak $Q$ nbd $A_{f}(s)$ of $P_{f}(s)$, there exists
a member $B_{f}(s)\in \beta $ such that $P_{f}(s)q_{w}B_{f}(s)\leq A_{f}(s)$.
\end{theorem}

\begin{proof}
The necessary part is straightforward. To prove its sufficiency, if possible
let $\beta $ be not a base for $\delta (s)$. Then there exists a member $%
A_{f}(s)\in \delta (s)-\beta $, such that $O_{f}(s)=\vee \{B_{f}(s)\in \beta 
$; $B_{f}(s)<A_{f}(s)\}\neq A_{f}(s)$, and hence there is an $x\in X$ and an 
$M\subset 
%TCIMACRO{\U{2115} }%
%BeginExpansion
\mathbb{N}
%EndExpansion
$ such that $O_{f}^{n}(x)<A_{f}^{n}(x)$, for all $n\in M$. Let $%
r=\{r_{n}\}_{n}$ where $r_{n}=1-O_{f}^{n}(x)>0$ whenever $n\in M$ and $%
r_{n}=0$ whenever $n\in 
%TCIMACRO{\U{2115} }%
%BeginExpansion
\mathbb{N}
%EndExpansion
-M$, then $A_{f}^{n}(x)+r_{n}>O_{f}^{n}(x)+r_{n}=1$, for all $n\in M$ and $%
(p_{fx}^{M}$, $r)=P_{f}(s)q_{w}A_{f}(s)$. Therefore $A_{f}(s)$ is an open
weak $Q$ nbd of $P_{f}(s)$. Now $B_{f}(s)=\{B_{f}^{n}\}_{n}\in \beta $, $%
B_{f}(s)\leq A_{f}(s)\Rightarrow B_{f}(s)<A_{f}(s)\Rightarrow B_{f}(s)\leq
O_{f}(s)\Rightarrow B_{f}^{n}(x)+r_{n}\leq O_{f}^{n}(x)+r_{n}=1$ for all $%
n\in M\Rightarrow P_{f}(s)\overline{q_{w}}B_{f}(s)$ which is a
contradiction. Hence the proof.
\end{proof}

\begin{proposition}
If $\beta $ be a base for the FST $\delta (s)$ on $X$, then $\beta
_{n}=\{B_{f}^{n}$; $B_{f}(s)=\{B_{f}^{n}\}_{n}\in \beta \}$ will form a base
for the component FT $\delta _{n}$ on $X$ for each $n\in 
%TCIMACRO{\U{2115} }%
%BeginExpansion
\mathbb{N}
%EndExpansion
$ but not conversely.

\begin{proof}
Proof of the first part is straightforward. For converse part we consider
the FSTS $(%
%TCIMACRO{\U{211d} }%
%BeginExpansion
\mathbb{R}
%EndExpansion
$, $\delta ^{%
%TCIMACRO{\U{2115} }%
%BeginExpansion
\mathbb{N}
%EndExpansion
})$, where $%
%TCIMACRO{\U{211d} }%
%BeginExpansion
\mathbb{R}
%EndExpansion
$ is the set of real numbers and $\delta =\{\overline{r}$; $r\in \lbrack 0$, 
$1]\}$, $\overline{r}(x)=r$ for all $x\in 
%TCIMACRO{\U{211d} }%
%BeginExpansion
\mathbb{R}
%EndExpansion
$, which is a FT on $%
%TCIMACRO{\U{211d} }%
%BeginExpansion
\mathbb{R}
%EndExpansion
$. Clearly $\beta _{n}=\{\overline{r}$; $r\in (0$, $1)\cap Q\}$, where $Q$
is the set of rational numbers, is a base for the component FT $\delta _{n}^{%
%TCIMACRO{\U{2115} }%
%BeginExpansion
\mathbb{N}
%EndExpansion
}$ on $X$ for each $n\in 
%TCIMACRO{\U{2115} }%
%BeginExpansion
\mathbb{N}
%EndExpansion
$ but $\beta (s)=\{X_{f}^{r}(s)$; $r\in (0$, $1)\cap Q\}$ is not a base for
the FST $\delta ^{%
%TCIMACRO{\U{2115} }%
%BeginExpansion
\mathbb{N}
%EndExpansion
}$ on $X$ because $A_{f}(s)=\{A_{f}^{n}\}_{n}$ where $A_{f}^{n}=\overline{(%
\frac{1}{n})}$ for all $n\in 
%TCIMACRO{\U{2115} }%
%BeginExpansion
\mathbb{N}
%EndExpansion
$ is a open fuzzy sequential set in $(%
%TCIMACRO{\U{211d} }%
%BeginExpansion
\mathbb{R}
%EndExpansion
$, $\delta ^{%
%TCIMACRO{\U{2115} }%
%BeginExpansion
\mathbb{N}
%EndExpansion
})$ but can not be written as a supremum of a subfamily of $\beta (s)$.
\end{proof}
\end{proposition}

\begin{definition}
Let $A_{f}(s)$ be any fuzzy sequential set in a FSTS $(X$, $\delta (s))$.
The closure $\overline{A_{f}(s)}$ and interior $\overset{o}{A_{f}}(s)$ of $%
A_{f}(s)$ are defined as%
\begin{equation*}
\overline{A_{f}(s)}=\wedge \{C_{f}(s)\text{; }A_{f}(s)\leq C_{f}(s)\text{, }%
C_{f}^{c}(s)\in \delta (s)\}\text{, }
\end{equation*}%
\begin{equation*}
\overset{o}{A_{f}}(s)=\vee \{O_{f}(s)\text{; }O_{f}(s)\leq A_{f}(s)\text{, }%
O_{f}(s)\in \delta (s)\}\text{.}
\end{equation*}
\end{definition}

\begin{proposition}
If $\overline{A_{f}(s)}=\{\overline{A_{f}^{n}}\}_{n}$ in $(X$, $\delta (s))$%
, then $cl(A_{f}^{n})\leq \overline{A_{f}^{n}}$ in $(X$, $\delta _{n})$ for
each $n\in 
%TCIMACRO{\U{2115} }%
%BeginExpansion
\mathbb{N}
%EndExpansion
$, where $cl(A_{f}^{n})$ is the closure of $A_{f}^{n}$ in $(X$, $\delta
_{n}) $.
\end{proposition}

\begin{proof}
Proof is straightforward.
\end{proof}

Here we cite such an example where the equality in the proposition $2.8$
does not hold. Let $X=[0$, $1]$ and $\delta (s)=\{X_{f}^{r}(s)$; $r\in
\lbrack 0$, $1]\}$. If $A_{f}(s)=P_{f}(s)=(p_{f\frac{1}{3}}^{%
%TCIMACRO{\U{2115} }%
%BeginExpansion
\mathbb{N}
%EndExpansion
}$, $r)$, $r=\{\frac{1}{2}-\frac{1}{3n}\}_{n}$, then $\overline{A_{f}(s)}%
=X_{f}^{\frac{1}{2}}(s)$. Here $cl(A_{f}^{n})=\overline{(\frac{1}{2}-\frac{1%
}{3n})}$, whereas $\overline{A_{f}^{n}}=\overline{(\frac{1}{2})}$.

\begin{definition}
The dual of a fuzzy sequential point $P_{f}(s)=(p_{fx}^{M}$, $r)$ is a fuzzy
sequential point $P_{df}(s)=(p_{fx}^{M}$, $t)$, where $r=\{r_{n}\}_{n}$, $%
t=\{t_{n}\}_{n}$ and%
\begin{eqnarray*}
t_{n} &=&1-r_{n}\text{ for all }n\in M\text{, } \\
&=&0\text{ for all }n\in 
%TCIMACRO{\U{2115} }%
%BeginExpansion
\mathbb{N}
%EndExpansion
-M\text{.}
\end{eqnarray*}
\end{definition}

\begin{theorem}
Every $Q$ nbd of a fuzzy sequential point $P_{f}(s)$ is weakly
quasi-coincident with a fuzzy sequential set $A_{f}(s)$ implies $P_{f}(s)\in 
\overline{A_{f}(s)}$ implies every weak $Q$ nbd of $P_{f}(s)$ and $A_{f}(s)$
are weakly quasi-coincident.

\begin{proof}
Let $P_{f}(s)=(p_{fx}^{M}$, $r)$. $P_{f}(s)\in \overline{A_{f}(s)}$ if for
every closed fuzzy sequential set $C_{f}(s)\geq A_{f}(s)$, $P_{f}(s)\in
C_{f}(s)$, that is $p_{f}^{n}(x)\leq c_{f}^{n}(x)$ for all $n\in
M\Longrightarrow P_{f}(s)\in \overline{A_{f}(s)}$ if for every open fuzzy
sequential set $B_{f}(s)=\{B_{f}^{n}\}_{n}\leq A_{f}^{c}(s)$, $%
B_{f}^{n}(x)\leq 1-p_{f}^{n}(x)$ for all $n\in M$; that is $P_{f}(s)\in 
\overline{A_{f}(s)}$ if for every open fuzzy sequential set $%
B_{f}(s)=\{B_{f}^{n}\}_{n}$ satisfying $B_{f}^{n}(x)>1-p_{f}^{n}(x)$ for all 
$n\in M$, $B_{f}(s)\nleq A_{f}^{c}(s)$, which implies the first part. Now
let $P_{f}(s)\in \overline{A_{f}(s)}$. If possible let there exists a weak $%
Q $ nbd $N_{f}(s)$ of $P_{f}(s)$ such that $N_{f}(s)\overline{q_{w}}A_{f}(s)$%
. Then there exists an open fuzzy sequential set $B_{f}(s)$\ such that $%
P_{f}(s)q_{w}B_{f}(s)\leq N_{f}(s)$. Now $N_{f}(s)\overline{q_{w}}A_{f}(s)$
and $B_{f}(s)\leq N_{f}(s)\Rightarrow B_{f}^{n}(x)+A_{f}^{n}(x)\leq 1$ for
all $x\in X$, $n\in 
%TCIMACRO{\U{2115} }%
%BeginExpansion
\mathbb{N}
%EndExpansion
\Rightarrow A_{f}(s)\leq B_{f}^{c}(s)\Rightarrow P_{f}(s)\in
B_{f}^{c}(s)\Rightarrow p_{f}^{n}(x)+B_{f}^{n}(x)\leq 1$ for all $n\in 
%TCIMACRO{\U{2115} }%
%BeginExpansion
\mathbb{N}
%EndExpansion
$. This contradicts the fact that $P_{f}(s)q_{w}B_{f}(s)$. Hence the result
follows.
\end{proof}
\end{theorem}

\begin{corollary}
A fuzzy sequential point $P_{f}(s)\in \overline{A_{f}(s)}$ if and only if
each nbd of its dual point $P_{df}(s)$ is weakly quasi-coincident with $%
A_{f}(s)$.
\end{corollary}

\begin{proof}
Proof is straightforward since $Q$ nbd of a fuzzy sequential point is
exactly the nbd of its dual point.
\end{proof}

\begin{theorem}
A fuzzy sequential point $P_{f}(s)\in \overset{o}{A}_{f}(s)$ if and only if
its dual point $P_{df}(s)\notin \overline{A_{f}^{c}(s)}$.
\end{theorem}

\begin{proof}
Let $P_{f}(s)\in \overset{o}{A}_{f}(s)$ $\Rightarrow $there exists an open
fuzzy sequential set $B_{f}(s)$ such that $P_{f}(s)\in B_{f}(s)\leq
A_{f}(s)\Rightarrow B_{f}(s)$ and $A_{f}^{c}(s)$ are not weakly
quasi-coincident$\Rightarrow P_{df}(s)\notin \overline{A_{f}^{c}(s)}$.
Conversely let $P_{df}(s)\notin \overline{A_{f}^{c}(s)}$. Then there exists
an open nbd $B_{f}(s)$ of $P_{f}(s)$ which is not weakly quasi-coincident
with $A_{f}^{c}(s)\Rightarrow P_{f}(s)\in B_{f}(s)\leq A_{f}(s)\Rightarrow
P_{f}(s)\in \overset{o}{A}_{f}(s)$.
\end{proof}

\begin{proposition}
In a FSTS $(X$,$\delta (s))$, the following hold:\newline
(i) $\overline{X_{f}^{r}(s)}=X_{f}^{r}(s)$, $r\in \{0$, $1\}$, (ii) $%
A_{f}(s) $ is closed if and only if $\overline{A_{f}(s)}=A_{f}(s)$, (iii) $%
\overline{\overline{A_{f}(s)}}=\overline{A_{f}(s)}$, (iv) $\overline{%
A_{f}(s)\vee B_{f}(s)}=\overline{A_{f}(s)}\vee \overline{B_{f}(s)}$, (v) $%
\overline{A_{f}(s)\wedge B_{f}(s)}\leq \overline{A_{f}(s)}\wedge \overline{%
B_{f}(s)}$, (vi) $(X_{f}^{r}(s))^{o}=X_{f}^{r}(s)$, $r\in \{0$, $1\}$, (vii) 
$A_{f}(s)$ is open if and only if $\overset{o}{A}_{f}(s)=A_{f}(s)$, (viii) $(%
\overset{o}{A}_{f}(s))^{o}=\overset{o}{A}_{f}(s)$, (ix) $(A_{f}(s)\wedge
B_{f}(s))^{o}=\overset{o}{A}_{f}(s)\wedge \overset{o}{B}_{f}(s)$, (x) $%
(A_{f}(s)\vee B_{f}(s))^{o}=\overset{o}{A}_{f}(s)\vee \overset{o}{B}_{f}(s)$%
, (xi) $\overset{o}{A}_{f}(s)=(\overline{A_{f}^{c}(s))}^{c}$, (xii) $%
\overline{A_{f}(s)}=\overline{(A_{f}^{c}(s))^{o}}$, (xiii) $(\overline{%
A_{f}(s)})^{c}=(A_{f}^{c}(s))^{o}$, $($xiv$)$ $\overline{(A_{f}^{c}(s))}=(%
\overset{o}{A}_{f}(s))^{c}$.
\end{proposition}

\begin{proof}
Proof is straightforward.
\end{proof}

\begin{definition}
A fuzzy sequential point $P_{f}(s)$ is called an adherence point of a fuzzy
sequential set $A_{f}(s)$ if and only if every weak $Q$ nbd of $P_{f}(s)$ is
weakly quasi-coincident with $A_{f}(s)$.
\end{definition}

\begin{definition}
A fuzzy sequential point $P_{f}(s)$ is called an accumulation point of a
fuzzy sequential set $A_{f}(s)$ if and only if $P_{f}(s)$ is an adherence
point of $A_{f}(s)$ and every weak $Q$ nbd of $P_{f}(s)$ and $A_{f}(s)$ are
weakly quasi-coincident at some sequential point having different base or
support from that of $P_{f}(s)$ whenever $P_{f}(s)\in A_{f}(s)$.
\end{definition}

\begin{proposition}
Any reduced sequential point of an accumulation point of a fuzzy sequential
set is also an accumulation point of it.
\end{proposition}

\begin{proof}
Proof is omitted.
\end{proof}

From the proposition $2.10$, we see that any simple reduced sequential point
of an accumulation point of a fuzzy sequential set is also an accumulation
point of it but the converse is not true. For let $X=\{a$, $b\}$ and $\delta
(s)=\{X_{f}^{r}(s)$, $G_{f}(s)$; $r\in \{0$, $1\}\}$ where $%
G_{f}(s)=\{G_{f}^{n}\}_{n}$, $G_{f}^{n}(a)=\frac{1}{2}$ $G_{f}^{n}(b)=0$ for 
$n\in 
%TCIMACRO{\U{2115} }%
%BeginExpansion
\mathbb{N}
%EndExpansion
$. Let $A_{f}(s)=\{A_{f}^{n}\}_{n}$ where $A_{f}^{n}=\overline{(\frac{2}{3})}
$ for $n=1$, $2$ and $A_{f}^{n}=0$ otherwise. Then the fuzzy sequential
point $P_{f}(s)=(p_{fa}^{M}$, $r)$ where $r=\{r_{n}\}_{n}$ with $r_{1}=r_{2}=%
\frac{2}{3}$ and $r_{n}=0$ otherwise, is not an accumulation point of $%
A_{f}(s)$ though $(p_{fa}^{1}$, $\frac{2}{3})$ and $(p_{fa}^{2}$, $\frac{2}{3%
})$ both are accumulation point of $A_{f}(s)$.

\begin{definition}
The union of all accumulation points of a fuzzy sequential set $A_{f}(s)$ is
called the fuzzy derived sequential set of $A_{f}(s)$ and it is denoted by $%
\overset{d}{A}_{f}(s)$.
\end{definition}

\begin{theorem}
In a FSTS $(X$, $\delta (s))$, $\overline{A_{f}(s)}=A_{f}(s)\vee \overset{d}{%
A}_{f}(s)$.
\end{theorem}

\begin{proof}
Let $\Omega =\{P_{f}(s)$; $P_{f}(s)$ is an adherence point of $A_{f}(s)\}$.
Then $\overline{A_{f}(s)}=\vee \Omega $. Now let $P_{f}(s)\in \Omega $ then
two cases may arise, $P_{f}(s)\in A_{f}(s)$ or $P_{f}(s)\notin A_{f}(s)$. If 
$P_{f}(s)\notin A_{f}(s)$ then $P_{f}(s)\in \overset{d}{A}_{f}(s)$, hence $%
P_{f}(s)\in A_{f}(s)\vee \overset{d}{A}_{f}(s)$. Therefore ,%
\begin{equation*}
\overline{A_{f}(s)}=\vee \Omega \leq A_{f}(s)\vee \overset{d}{A}%
_{f}(s).-----------(1)\text{.}
\end{equation*}%
Again, $A_{f}(s)\leq \overline{A_{f}(s)}$ and since any accumulation point $%
P_{f}(s)$ of $A_{f}(s)$ belongs to $\overline{A_{f}(s)}$ which implies $%
\overset{d}{A}_{f}(s)\leq \overline{A_{f}(s)}$. Therefore ,%
\begin{equation*}
A_{f}(s)\vee \overset{d}{A}_{f}(s)\leq \overline{A_{f}(s)}-----------(2)%
\text{.}
\end{equation*}%
From (1) and (2) the result follows.
\end{proof}

\begin{corollary}
A fuzzy sequential set is closed in a FSTS $(X$, $\delta (s))$ if and only
if it contains all its accumulation points.
\end{corollary}

\begin{proof}
Proof is straightforward.
\end{proof}

\begin{remark}
The fuzzy derived sequential set of any fuzzy sequential set may not be
closed as shown by example $2.1$.
\end{remark}

\begin{example}
Let $X=\{a$, $b\}$, $\delta (s)$ be the FST having base $\beta
=\{X_{f}^{1}(s)\}\vee \{$ $X_{f}^{0}(s)\}\vee \{P_{f}(s)$, $G_{f}(s)\}$,
where $G_{f}^{n}(b)=1$ $\forall $ $n\in 
%TCIMACRO{\U{2115} }%
%BeginExpansion
\mathbb{N}
%EndExpansion
$, $G_{f}^{n}(a)=0$ $\forall $ $n\in 
%TCIMACRO{\U{2115} }%
%BeginExpansion
\mathbb{N}
%EndExpansion
$ and $P_{f}(s)=(p_{fa}^{M}$, $r)$, where $M=\{1$, $2$, $3\}$, $r_{1}=0.5$, $%
r_{2}=1$, $r_{3}=0.3$, $r_{n}=0$ $\forall $ $n\neq 1$, $2$, $3$. Here the
fuzzy derived sequential set of $(p_{fa}^{3}$, $0.3)$ is not closed.
\end{example}

\begin{proposition}
The fuzzy derived sequential set of a fuzzy sequential point equals the
union of the fuzzy derived sequential sets of all its simple reduced fuzzy
sequential points.
\end{proposition}

\begin{proof}
The proof is omitted.
\end{proof}

\begin{proposition}
If the fuzzy derived sequential set of each of the simple reduced fuzzy
sequential points of a fuzzy sequential point is closed, then the derived
sequential set of the fuzzy sequential point is closed.
\end{proposition}

\begin{proof}
Let $A_{f}(s)=(p_{fx}^{M}$, $r)$ be a fuzzy sequential point. Let $D_{f}(s)$
be the fuzzy derived sequential set of $A_{f}(s)$. Let $D_{nf}(s)$ be the
fuzzy derived sequential set of $A_{nf}(s)=(p_{fx}^{n}$, $r_{n})$, $n\in M$.
Suppose $D_{nf}(s)$ is closed for all $n\in M$. Let $P_{f}(s)$ be an
accumulation point of $D_{f}(s)$.\newline
Now, $P_{f}(s)\notin D_{f}(s)$\newline
$\Longrightarrow P_{f}(s)$ is not an accumulation point of $(p_{fx}^{M}$, $%
r) $\newline
$\Longrightarrow \exists $ a weak Q-nbd $B_{f}(s)$ of $P_{f}(s)$ which is
not weakly quasi coincident with $(p_{fx}^{M}$, $r)$\newline
$\Longrightarrow B_{f}(s)$ is not weakly quasi coincident with $(p_{fx}^{n}$%
, $r_{n})$ $\forall $ $n\in M$.\newline
$\Longrightarrow P_{f}(s)\notin D_{nf}(s)$ $\forall $ $n\in M$\newline
$\Longrightarrow P_{f}(s)$ is not an accumulation point of $D_{nf}(s)$ $%
\forall $ $n\in M$ (since $D_{nf}(s)$ is closed $\forall $ $n\in M$)\newline
$\Longrightarrow P_{f}(s)$ is not an accumulation point of $\vee _{n\in
M}D_{nf}(s)=D_{f}(s)$, a contradiction. Hence proved.
\end{proof}

\begin{remark}
Converse of proposition $2.12$ is not true as shown by example $2.2$.
\end{remark}

\begin{example}
Let $X=\{a$, $b\}$, $\delta (s)$ be the FST having base $\beta
=\{X_{f}^{1}(s)\}\vee \{$ $X_{f}^{0}(s)\}\vee \{P_{f}(s)$, $G_{f}(s)\}$,
where $G_{f}^{n}(b)=1$ $\forall $ $n\in 
%TCIMACRO{\U{2115} }%
%BeginExpansion
\mathbb{N}
%EndExpansion
$, $G_{f}^{n}(a)=0$ $\forall $ $n\in 
%TCIMACRO{\U{2115} }%
%BeginExpansion
\mathbb{N}
%EndExpansion
$ and $P_{f}(s)=(p_{fa}^{M}$, $r)$, where $M=\{1$, $2$, $3\}$, $r_{1}=0.5$, $%
r_{2}=1$, $r_{3}=0.3$, $r_{n}=0$ $\forall $ $n\neq 1$, $2$, $3$. Here the
fuzzy derived sequential set of $P_{f}(s)$ is closed but the fuzzy derived
sequential set of $(p_{fa}^{3}$, $0.3)$ is not closed.
\end{example}

\begin{lemma}
Let $A_{f}(s)=(p_{fx}^{k}$, $r)$ be a fuzzy sequential point in FSTS $(X$, $%
\delta (s))$. Then,\newline
(i) For $y\neq x$, $\overline{A_{f}(s)}(y)=A_{f}^{d}(s)(y)$.\newline
(ii) If $\overline{A_{f}(s)}(x)>_{P}r$, $\overline{A_{f}(s)}%
(x)=_{P}A_{f}^{d}(s)(x)$, where $P\subset M$.\newline
(iii) If $\overline{A_{f}(s)}(x)>_{M}r$, $\overline{A_{f}(s)}%
(x)=A_{f}^{d}(s)(x)$.\newline
(iv) If $A_{f}^{d}(s)(x)=0=$sequence of real zeros, then $\overline{A_{f}(s)}%
(x)=r$.\newline
(v) If $A_{f}(s)$ is simple then converse of (iv) is true.
\end{lemma}

\begin{lemma}
Let $A_{f}(s)=(p_{fx}^{k}$, $r_{k})$ be a simple fuzzy sequential point in
FSTS $(X$, $\delta (s))$. Then,\newline
(i) If $A_{f}^{d}(s)(x)$ is a non zero sequence, then $\overline{A_{f}(s)}%
=A_{f}^{d}(s)$.\newline
(ii) If $A_{f}^{d}(s)(x)=0=$sequence of real zeros, then $A_{f}^{d}(s)$ is
closed iff $\exists $ an open fuzzy sequential set $B_{f}^{@}(s)$ such that $%
B_{f}^{@}(s)(x)=1$ and for $y\neq x$, $B_{f}^{@}(s)(y)=\{\overline{A_{f}(s)}%
\}^{c}(y)=\{A_{f}^{d}(s)\}^{c}(y)$.\newline
(iii) $A_{f}^{d}(s)(x)=0=$sequence of real zeros iff $\exists $ an open
fuzzy sequential set $B_{f}(s)$ such that $B_{f}(s)(x)=1-r$ $where$ $%
r=\{r_{n}\}_{n}$ and $r_{n}=0$ if $n\neq k$, $r_{n}=r_{k\text{ }}$ if n=k.
\end{lemma}

\begin{theorem}
The fuzzy derived sequential set of each fuzzy sequential set is closed iff
the fuzzy derived sequential set of each simple fuzzy sequential point is
closed.
\end{theorem}

\begin{proof}
The necessity is obvious. Conversely, suppose $H_{f}(s)$ is a fuzzy
sequential set. We will show that $H_{f}^{d}(s)=D_{f}(s)$ is closed. Let $%
P_{f}(s)=(p_{fx}^{k}$, $r_{k})$ be an accumulation point of $D_{f}(s)$. It
is sufficient to show that $P_{f}(s)\in D_{f}(s)$. Let $r=\{r_{n}\}_{n}$
where $r_{n}=r_{k}$ for $n=k$ and $r_{n}=0$ $\forall $ $n\neq k$. Now $%
P_{f}(s)\in \overline{D_{f}(s)}=\overline{H_{f}^{d}(s)}\leq \overline{%
\overline{H_{f}(s)}}=\overline{H_{f}(s)}$. Therefore $P_{f}(s)$ is an
adherence point of $H_{f}(s)$. If $P_{f}(s)\notin H_{f}(s)$, then $P_{f}(s)$
is an accumulation point of $H_{f}(s)$, that is $P_{f}(s)\in D_{f}(s)$ and
we are done.\newline
Let us assume $P_{f}(s)\in H_{f}(s)$

$\ \ \ \Longrightarrow r\leq H_{f}(s)(x)=\rho $ (say)

\ \ \ \ $\Longrightarrow r_{k}\leq H_{f}^{k}(x)=\rho _{k}$\newline
Now consider the simple fuzzy sequential point $A_{f}(s)=(p_{fx}^{k}$, $\rho
_{k})$. Let $\rho ^{\prime }=\{\rho _{n}^{\prime }\}_{n}$ where $\rho
_{k}^{\prime }=\rho _{k}$ and $\rho _{n}^{\prime }=0$ $\forall $ $n\neq k$.
There are two possibilities \ concerning $A_{f}^{d}(s)$.\newline
Case I. $A_{f}^{d}(s)(x)=\rho _{1}$ is a non zero sequence. Now $\overline{%
A_{f}(s)}(x)\geq A_{f}(s)(x)=\rho ^{\prime }$\newline
By lemma $2.1(v)$, $\overline{A_{f}(s)}(x)>\rho ^{\prime }$

$\ \ \ \ \ \ \Longrightarrow A_{f}^{d}(s)(x)=\overline{A_{f}(s)}(x)>\rho
^{\prime }$

\ \ \ \ \ \ \ $\Longrightarrow \rho _{1}>\rho ^{\prime }$

\ \ \ \ \ \ \ $\Longrightarrow \rho _{1k}>\rho
_{k}=A_{f}^{k}(x)=H_{f}^{k}(x) $\newline
Hence the simple fuzzy sequential point $Q_{f}(s)=(p_{fx}^{k}$, $\rho
_{1k})\notin H_{f}(s)$. but since $Q_{f}(s)\in A_{f}^{d}(s)\leq \overline{%
A_{f}(s)}\leq \overline{H_{f}(s)}$, $Q_{f}(s)$ is an accumulation point of $%
H_{f}(s)$, that is $Q_{f}(s)\in D_{f}(s)$. Moreover $r_{k}\leq \rho
_{k}<\rho _{1k}$

$\ \ \ \ \ \ \ \ \ \ \ \ \ \ \ \ \ \ \ \ \ \ \ \ \ \ \ \ \ \ \Longrightarrow 
$ $r_{k}<\rho _{1k}$

$\ \ \ \ \ \ \ \ \ \ \ \ \ \ \ \ \ \ \ \ \ \ \ \ \ \ \ \ \ \ \Longrightarrow
P_{f}(s)\in D_{f}(s)$.\newline
Case II. $A_{f}^{d}(s)(x)=0$. Let $B_{f}(s)$ be an arbitrary weak Q-nbd of $%
A_{f}(s)$ and hence of $P_{f}(s)$. In view of lemma $2.2(ii)$, $\exists $ an
open fuzzy sequential set $B_{f}^{@}(s)$ such that $B_{f}^{@}(s)(x)=1$ and
for $y\neq x$, $B_{f}^{@}(s)(y)=\{\overline{A_{f}(s)}\}^{c}(y)$. Let $%
C_{f}(s)=B_{f}(s)\wedge B_{f}^{@}(s)$. Then $C_{f}(s)(x)=B_{f}(s)(x)$ which
implies $C_{f}^{k}(x)=B_{f}^{k}(x)>1-r_{k}$. Thus $C_{f}(s)$ is a weak Q-nbd
of $P_{f}(s)$. Hence $C_{f}(s)$ and $D_{f}(s)$ are weakly quasi coincident,
that is $\exists $ a point $z$ and $n\in 
%TCIMACRO{\U{2115} }%
%BeginExpansion
\mathbb{N}
%EndExpansion
$ such that $D_{f}^{n}(z)+C_{f}^{n}(z)>1$. Owing to the fact that $D_{f}(s)$
is the union of all the accumulation points of $H_{f}(s)$, $\exists $ an
accumulation point $P_{f}^{\prime }(s)=(p_{fz}^{n}$, $\mu _{n})$ such that $%
\mu _{n}+C_{f}^{n}(z)>1$. Therefore $C_{f}(s)$ is a weak Q-nbd of $%
P_{f}^{\prime }(s)$. Let $\mu =\{\mu _{n}\}_{n}$ where $\mu _{n}\neq 0$ and $%
\mu _{m}=0$ $\forall $ $m\neq n$. The proof will be carried out, according
to the following subcases:\newline
Subcase I. When n=k.\newline
(a) when $z=x$ and $\mu \leq \rho ^{\prime }$, then $P_{f}^{\prime }(s)\in
H_{f}(s)$. Since $P_{f}^{\prime }(s)$ is an accumulation point of $H_{f}(s)$%
, every weak Q-nbd of $P_{f}^{\prime }(s)$ (and hence $B_{f}(s)$) and $%
H_{f}(s)$ are weakly quasi coincident at some point having different base or
different support than that of $P_{f}(s)$.\newline
(b) When $z=x$ and $\mu >\rho ^{\prime }$, then $P_{f}^{\prime }(s)\notin
H_{f}(s)$. From lemma $2.2(iii)$, $\exists $ an open fuzzy sequential set $%
B_{f}^{\prime }(s)$ such that $B_{f}^{\prime }(s)(x)=1-\rho ^{\prime }>1-\mu 
$. Therefore $G_{f}(s)=C_{f}(s)\wedge B_{f}^{\prime }(s)$ is also a weak
Q-nbd of $P_{f}^{\prime }(s)$. Hence $G_{f}(s)$ and $H_{f}(s)$ are weakly
quasi coincident. Since $G_{f}(s)(x)\leq B_{f}^{\prime }(s)(x)=1-\rho
^{\prime }$

$\ \ \ \ \ \ \ \ \ \ \ \ \ \ \ \ \ \ \ \ \ \ \ \ \ \ \ \ \ \ \ \ \
\Longrightarrow G_{f}^{k}(x)\leq B_{f}^{\prime k}(x)=1-\rho _{k}$\newline
Thus $G_{f}(s)$ (and hence $B_{f}(s)$) and $H_{f}(s)$ are weakly quasi
coincident at some point having different base or different support than
that of $P_{f}(s)$.\newline
(c) When $z\neq x$.\newline
We have $B_{f}^{@}(s)(z)=\{\overline{A_{f}(s)}\}^{c}(z)$. Also $\{\overline{%
A_{f}(s)}\}^{c}=\{A_{f}^{c}(s)\}^{\circ }$. Since $\{A_{f}^{c}(s)\}^{\circ
}(z)=B_{f}^{@}(s)(z)\geq C_{f}(s)(z)$,$~\exists $ an open fuzzy sequential
set $B_{f}^{\prime \prime }(s)\leq A_{f}^{c}(s)$ such that $B_{f}^{\prime
\prime k}(z)\geq C_{f}^{k}(z)(z)>1-\mu _{k}$. Therefore $G_{f}^{\prime
}(s)=B_{f}(s)\wedge B_{f}^{\prime \prime }(s)$ is also a weak Q-nbd of $%
P_{f}^{\prime }(s)$ and hence is weakly quasi coincident with $H_{f}(s)$.%
\newline
Since $B_{f}^{\prime \prime }(s)$ $\leq A_{f}^{c}(s)$\newline
$\Longrightarrow B_{f}^{\prime \prime }(s)(x)\leq 1-A_{f}(s)(x)$\newline
$\Longrightarrow B_{f}^{\prime \prime k}(x)\leq
1-A_{f}^{k}(x)=1-H_{f}^{k}(x) $.\newline
Thus $G_{f}^{\prime }(s)$ (and hence $B_{f}(s)$) is weakly quasi coincident
with $H_{f}(s)$ at some point having different base or different support
than that of $P_{f}(s)$.\newline
Subcase II. When n$\neq $k.\newline
(a) Suppose $z=x$. We have $B_{f}^{\prime }(s)(x)=1-\rho ^{\prime }$

$\Longrightarrow B_{f}^{\prime n}(x)=1>1-\mu _{n}$\newline
So $B_{f}^{\prime }(s)$ is a weak Q-nbd of $P_{f}^{\prime }(s)$. Hence $%
G_{f}(s)=C_{f}(s)\wedge B_{f}^{\prime }(s)$ is a weak Q-nbd of $%
P_{f}^{\prime }(s)$ and so it is weakly quasi coincident with $H_{f}(s).$%
\newline
Now $G_{f}(s)(x)\leq B_{f}^{\prime }(s)(x)=1-\rho ^{\prime }$.\newline
$\Longrightarrow $ $G_{f}^{k}(x)\leq B_{f}^{\prime k}(x)=1-\rho
_{k}=1-H_{f}^{k}(x)$.\newline
So $H_{f}(s)$ and $G_{f}(s)$ are weakly quasi coincident at some point
having different base or different support than that of $P_{f}(s)$.\newline
(b) When $z\neq x$, the proof is same as Subcase I (c).
\end{proof}

\section{Acknowledgement:}

The second Author is thankful to the Council of Scientific and Industrial
Research (CSIR), New Delhi, India, for the financial assistance awarded to
her through the NET-JRF program.\newpage

\section{References}

[1] L.A. Zadeh, "Fuzzy Sets", Information And Control 8, 338-353(1965).%
\newline
[2] M.K. Bose and Indrajit Lahiri, "Sequential Topological Spaces and
Separation Axioms", Bulletin of The Allahabad Mathematical Society,
17(2002), 23-37.\newline
[3] N. Palaniappan, "Fuzzy Topology", Narosa Publishing House, New Delhi
(2006).\newline
[4] N. Tamang, M. Singha and S. De Sarkar, "Separation Axioms in Sequential
Topological Spaces in the Light of Reduced and Augmented Bases", Int. J.
Contemp. Math. Sciences, Vol. 6, 2011, no.23, 1137-1150.

\end{document}